\documentclass[12pt]{amsart} 
\usepackage[english]{babel}
\usepackage{amsmath}
\usepackage{changes}
\usepackage{amsthm}
\usepackage{amsfonts} 
\usepackage{epsfig}
\usepackage{amssymb}
\usepackage{mathrsfs}
\usepackage{tikz, pgfplots}

\pgfplotsset{compat=1.14}
\usepgfplotslibrary{fillbetween}
\usetikzlibrary{patterns}

\numberwithin{equation}{section}

\allowdisplaybreaks

\newcommand{\eqa}{\begin{eqnarray}}
\newcommand{\ena}{\end{eqnarray}}
\newcommand{\eq}{\begin{equation}}
\newcommand{\en}{\end{equation}}
\newcommand{\eqs}{\begin{eqnarray*}}
\newcommand{\ens}{\end{eqnarray*}}

\def\X{\mathbf{X}} 
 
\newcommand{\Z}     {\mathbb{Z}} 
\newcommand{\N}     {\mathbb{N}} 
\renewcommand{\P}   {\mathbb{P}} 
\newcommand{\E}     {\mathbb{E}}

\newcommand{\M}    {\mathbf{M}} 

 \newcommand{\floor}[1]{\left\lfloor #1 \right\rfloor}

\def\1{{\mathchoice {1\mskip-4mu\mathrm l}      
{1\mskip-4mu\mathrm l} 
{1\mskip-4.5mu\mathrm l} {1\mskip-5mu\mathrm l}}} 
\newcommand{\ssup}[1] {{{\scriptscriptstyle{({#1}})}}} 
\def\comment#1{} 
 
\renewcommand{\d}{{\rm d}} 
 
\newcommand{\eps}{\varepsilon}

\newcommand{\Gcal}   {{\mathcal G }}

\def\ignore#1{}

\def\Def{:=} 

\newtheorem{theorem}{Theorem}
\newtheorem{proposition}[theorem]{Proposition}
\newtheorem{lemma}[theorem]{Lemma}
\newtheorem{remark}[theorem]{Remark}
\newtheorem{example}[theorem]{Example}

\newtheorem{definition}[theorem]{Definition}

\newtheorem*{ack}{Acknowledgement}

\newtheorem{theoAlph}{Theorem}

\renewcommand{\epsilon}{\varepsilon}

\title[Phase transitions for  ERRW on $\Z_+$]{Phase transitions for edge-reinforced random walks on the half-line}
\date{\today}
\author[J.~Akahori]{Jiro Akahori}
\address{Jiro Akahori\\ Department of Mathematical Sciences, Ritsumeikan University} \email{akahori@se.ritsumei.ac.jp}
\author[A.~Collevecchio]{Andrea Collevecchio}
\address{Andrea Collevecchio\\ School of Mathematical Sciences, Monash  University, Melbourne} \email{andrea.collevecchio@monash.edu}
\author[M.~Takei]{Masato Takei}
\address{Masato Takei\\ Department of Applied Mathematics, Faculty of Engineering, Yokohama National University} \email{takei-masato-fx@ynu.ac.jp}

\newcommand{\MARU}[1]{{\ooalign{\hfil#1\/\hfil\crcr\raise.167ex\hbox{\mathhexbox20D}}}}

\date{}

\begin{document}

\begin{abstract} We study the behaviour of a class of edge-reinforced random walks {on $\mathbb{Z}_+$}, with heterogeneous initial weights, where each edge weight can be updated only when the edge is traversed from left to right. We provide a description for different behaviours of this process and describe phase transitions that arise as trade-offs between the strength of the reinforcement and that of the initial weights. Our result aims to complete the ones given by Davis~\cite{Davis89, Davis90},  Takeshima~\cite{Takeshima00, Takeshima01} and Vervoort~\cite{Vervoort00}.
\end{abstract}

\maketitle

\section{Introduction}
Reinforced random walks (RRW) have been extensively studied in  the past 30 years. 
The canonical model is the one introduced by Coppersmith and  Diaconis \cite{CD}, called Linearly Edge-Reinforced  Random Walk (LERRW) which can be described as follows. Consider a graph $\Gcal$ which is locally finite and to each edge assign initial weight one. These weights are updated depending on the behaviour of the process. LERRW takes values on the vertices of $\Gcal$, at each step it jumps to vertices which are neighbors of the present one, say $x$. The probability to pick a particular neighbor is proportional to the weight of the edge connecting that vertex to $x$. Each time the process traverses an edge, its weight is increased by one.    When $\Gcal$ is a tree, then LERRW is a random walk in an i.i.d. environment. {In general, it can be represented as a mixture of Markov chains (see Merkl and Rolles \cite{MR07b}).} The {mixing} measure is connected with  $H^{2/2} $ models, which in turn are used to explain the phenomena of Anderson localization. {For more information about this connection and details about  $H^{2/2} $ models, see for example \cite{ST15} and its bibliography.}

Our goal is to study a large class of edge-reinforced walks on $\Z_+$, inspired by the {work} of Davis~\cite{Davis89,Davis90}. We allow heterogeneous initial weights on the edges, and a  reinforcement that is different from linear. 
{A theorem of Vervoort~\cite{Vervoort00} establishes an interesting recurrence criterion for a large class of RRW with general  initial weights. The reinforcement scheme of these processes is characterised by the fact that there is a chance that an edge increases its weight  when traversed from left to right (see Theorem~\ref{thm:Vervoort00Z+}  below for a precise statement).} Hence, we focus our attention on the case where the reinforcement can happen only when the process traverses an edge from right to left. Intuitively, this class of {processes} is the \lq most\rq\ transient   and has an  interesting phase transition in terms of the initial weights and the reinforcement. 

We provide a general phase diagram for  edge-reinforced  random walks which take values on the vertices of $\Z_+$ and heterogenous initial conditions. This includes a description of  phase transitions that are trade-offs between the strength of the reinforcement and that of the initial weights. We use a martingale approach, a theorem of Austin (see  \cite{Austin66}) and the so-called Rubin construction (see \cite{Davis90}) combined with Cram\'er-type bounds. {Some of the methods} used in the proofs are close in spirit to the ones proposed by Davis {in \cite{Davis89, Davis90}}.

\subsection{Edge-reinforced random walks on the half-line}

We define the {\it edge-reinforced random walk} (ERRW), denoted by $\mathbf{X} =\{X_n\}_n $, as follows. This process takes values on the vertices of $\Z_+$ and at each step it jumps to one of the nearest neighbors. 
Denote by $\{x,x+1\}$ the  non-oriented edge connecting $x $ and  $x+1$. In contrast, we use $(x, y)$  to denote the oriented edge connecting $x$ to $y$. Define
\begin{equation}\label{def:phi}
 \phi_n(x) \Def \sum_{i=1}^n { \1_{ \{X_{i-1},\,X_i\} = \{x,\,x+1\} }}, 
 \end{equation}
that is the number of traversals of the edge $\{x,\,x+1\}$ by time $n$.
For each $x \in \mathbb{Z}_+$, let {$\mathbf{f}_x =(f(\ell,x) \colon \ell \in \Z_+)$} be a non-decreasing sequence of positive numbers, called {the} {\it reinforcement scheme} at $x \in \mathbb{Z}_+$. For each $n \geq 0$, the weights at time $n$ are defined by  
\[
w_n (x) = f(\phi_n(x),x)\quad \mbox{for $x \in \mathbb{Z}_+$}, 
\]
and
the transition probability is given by
\begin{align*}
\P(X_{n+1}=X_n+1\,|\,X_0,\ldots,X_n)&=1-\P(X_{n+1}=X_n-1\,|\,X_0,\ldots,X_n) \\
 &=\frac{w_n (X_n)}{w_n(X_n-1)+w_n(X_n)}.
\end{align*}
Here we set $w_n(-1) = 0$ for all $n \in \Z_{+}$, which implies   a reflection at the origin.

We say that the path $\mathbf{X}(\omega)$ is  {\it recurrent} if every point is visited infinitely often,  and
{\it transient} if every point is visited only finitely many times. Finally, if the set $R$   of points that $\mathbf{X}(\omega)$ visits infinitely often is finite, then we say that $\mathbf{X}(\omega)$ {\it localizes}. {Takeshima \cite{Takeshima00}} proved that {ERRW} $\X(\omega)$ {on $\Z_+$} can be either recurrent, transient or it localizes.  Notice there are cases where 
$$ 0< \P(\mathbf{X} \mbox{ is transient})<1.$$
{In fact, if we set $f(\ell, 0) = (\ell+1)^2$ for all $\ell \in \Z_+$, and $f(\ell, x) = x^2$ for all $\ell \in \Z_+$ and $x \in \N$, we have the following.  The process does not visit $2$, i.e. only visits the vertices $0$ and $1$, with probability
$$ \prod_{k=1}^\infty \frac{4k^2}{1 + 4k^2} >0.$$
Moreover, with positive probability the process drifts away to infinity. In fact, it  behaves like a transient Markov chain on the sites $x$ with $x \ge 2$.} 

{ For $\ell,  x \in \mathbb{Z}_+$  and  $k \in \N$,} define
\[
F_\ell ^{\ssup k}= \sum_{y=0}^{\infty} \left(\frac{1}{f(\ell,y)}\right)^k, \qquad { \Phi_x =  \sum_{j=0}^{\infty} \dfrac{1}{f(j,x)}}.
 \]  

It is well known 
that if  $f(\ell,x)=f(0,x)$ for all $\ell \geq 0$ and $x \in \mathbb{Z}_+$, i.e. if there is no reinforcement at all, then 
$\X$ is recurrent  a.s. if $F_0^{\ssup 1}=+\infty$, and is transient a.s. otherwise. 
In fact, $F_0^{\ssup 1}$ can be associated to the effective resistance of the network, which characterizes the behaviour of the relative Markov chain (see \cite{LP}). 
We say the ERRW is {\it initially recurrent} (resp. {\it initially transient})   if $F_0^{\ssup 1}=+\infty$  (resp. $F_0^{\ssup 1}<+\infty$).

Davis  \cite{Davis89} proved that initially recurrent reinforced random  walks are not necessarily recurrent. 
\setcounter{theoAlph}{3}
\begin{theoAlph}[Davis \cite{Davis89}]\label{Da} Consider the ERRW $\mathbf{X}$ on $\mathbb{Z}_+$.
\begin{itemize}
 \item[(i)] If $F_0^{\ssup 2}=+\infty$, then
 $\mathbf{X}$ is either recurrent or it localizes on a single edge.
 \item[(ii)] There exists a reinforcement scheme $(\mathbf{f}_x \colon x \in \Z_+)$ such that   $F_0^{\ssup 2}<+\infty$ and $F_0^{\ssup 1}=+\infty$ and $\X$ is transient with positive probability.
\end{itemize}
\end{theoAlph}

The known phases {can be summarised} in the following table, where we combined   Theorem \ref{Da}  with other  results by {Davis \cite{Davis90}, Sellke \cite{Sellke94} and Takeshima \cite{Takeshima00,Takeshima01}.}
\begin{center}
 {\bf TABLE I}
 \vspace{0.3cm}
 
\begin{tabular}{|c||c|c|c|} \hline
 &  \multicolumn{2}{|c|}{$F_0^{\ssup 1}= +\infty$} &  $F_0^{\ssup 1} < +\infty$ \\ \cline{2-3} 
 & $F_0^{\ssup 2}=  +\infty$ & $F_0^{\ssup 2}<  +\infty$ & \\ \hline \hline
 $\forall x \in \mathbb{Z}_+$, $\Phi_x = +\infty$ & recurrent a.s.  & ?? & transient  a.s. \\[1mm] \hline
 $\exists x \in \mathbb{Z}_+$, $\Phi_x < +\infty$ & \begin{minipage}{0.2\hsize}\begin{center} localizes on \\
 one edge a.s. \end{center}\end{minipage} & ?? & ??
  \\[4mm] \hline
\end{tabular}\\

\end{center}

\noindent

The question marks in Table I  indicate what is left open in general. In this paper we partially fill these gaps  for a general class of reinforcement, which we call Factor Type Reinforcement (FTR), and which is of the form
\begin{equation}
f(\ell,x) = \delta_{\ell} \cdot f(0,x), \qquad \mbox{for all $\ell\in \mathbb{Z}_+$},
\end{equation}
where {$\boldsymbol{\delta}=(\delta_\ell\colon \ell \in \Z_+)$} is  a positive non-decreasing sequence with $\delta_0 = 1$. Furthermore, Vervoort proved the following result  (Theorem 8.2.2 in \cite{Vervoort00}). For the sake of completeness we include the proof in the Appendix. \setcounter{theoAlph}{21}
\begin{theoAlph}[Vervoort \cite{Vervoort00}] \label{thm:Vervoort00Z+} Suppose that $F_0^{\ssup 1}=+\infty$, and suppose that $\X$ has FTR. Then, the  process $\X$ is recurrent a.s. if either i) {$\boldsymbol{\delta}$} is bounded or ii)  
 $ \Phi_x=+\infty$ for all $x \in \Z_+$, and $\delta_{2k}<\delta_{2k+1}$ for {some $k\in \Z_+$}.
\end{theoAlph}
{By virtue of} Theorem~\ref{thm:Vervoort00Z+}, we can focus on the case where {$\boldsymbol{\delta}$} is unbounded and an edge can be reinforced only when the process traverses from right to left, i.e. {when $\delta_{2k} = \delta_{2k+1}$ for all $k  \in \Z_+$}.

\section{Main results}
\begin{definition} The sequence $\boldsymbol{\delta}$ is called down-only type (DT) if $\delta_{2k} = \delta_{2k+1}$, if it is non-decreasing,  and $\delta_0 = 1$.
\end{definition}
\begin{figure}[h]
	\centering
	\begin{tikzpicture}
		\begin{axis}[ymin=0, ymax=1.2, xmin=0, xmax=1.2, xlabel=$\alpha$, ylabel=$\rho$, area legend, legend pos=outer north east]
			\addplot[green!70, fill=green!70] coordinates {(0,0) (0,1) (0.5,1) (0.5,0.5) (1,0) (0,0)};
			\addlegendentry{Recurrent a.s.}
			\addplot[blue!70, fill=blue!70] coordinates {(1,0) (1,1) (1.2,1) (1.2,0) (1,0)};
			\addlegendentry{Transient a.s.}
			\addplot[red!70, fill=red!70] coordinates {(0,1) (0,1.2) (1.2,1.2) (1.2,1) (0,1)};
			\addlegendentry{Localizes a.s.}
			\addplot[red!70, domain=0.5:1, draw=none, forget plot, name path=A] {1};
			\addplot[blue!70, domain=0.5:1, draw=none, forget plot, name path=B] {(1.5 - x)/(2.5 - x)};
			\addplot[green!70, domain=0.5:1, draw=none, forget plot, name path=C] {1 - x};
			\addplot[blue!70, forget plot] fill between [of=A and B];
			\addplot[pattern=north east lines, pattern color=black] fill between [of=B and C];
	              \addlegendentry{Unknown}
		\end{axis}
	\end{tikzpicture}
	\vspace{-.6 cm}
	\begin{caption}
	{\small Different phases in Theorem~\ref{thm:ACT18-main}}
	\end{caption}
\end{figure}

Our main result is the following.

\begin{theorem} \label{thm:ACT18-main} Let $\X$ be a reinforced random walk with FTR, and suppose that ${\boldsymbol \delta}$ is DT.  Let $f(0,x)=(x+1)^{\alpha}$ and $\delta_{2k}=(k+1)^{\rho}$, for all $x, k \in \Z_+$,  with $ \alpha \in (1/2, 1]$ and $\rho \in [0,\infty)$.
\begin{itemize}
\item[1)] If $\rho  \in [0, 1-\alpha]$, then $\mathbf{X}$ is recurrent a.s..
\item[2)] If $\rho \in (1-\alpha,1/2]$ and $\rho > (1.5 - \alpha)/(2.5 - \alpha)$, then $\mathbf{X}$ is transient a.s..
\item[3)] If $\rho \in (1/2, 1]$, then $\mathbf{X}$ is transient a.s.. 
\item[4)] If $\rho \in  (1, \infty)$, then  $\mathbf{X}$ localizes on a single edge a.s..
\end{itemize}
\end{theorem}
{As highlighted in the next example, if we perturb a single reinforcement even slightly, we can witness a  transition  from recurrence to transience.}
\begin{example}
Let $\delta_{2k} = \delta_{2k+1} = (k+1)^{0.4}$, for all $k \in \Z_+$.~Let $f(0, x) = (x+1)^{0.9}$, for all $x \in   \Z_+$. For  $\eps \in (0, 1)$, define the family of functions
$$g_\eps (2k) = 
\begin{cases}
0 \qquad \mbox{if $k \neq \floor{1/\eps}$},\\
\eps \qquad \mbox{if $k = \floor{1/\eps}$},
\end{cases}
$$
and $g_\eps(2k+1) = 0$ for all $k \in \Z_+$. Define  the family of reinforced random walks  $\X^{\ssup \eps}$ with FTR  {$f_\eps(\ell, x) =  \delta^{\ssup \eps}_\ell f(0, x)$ where  $\delta^{\ssup \eps}_\ell =\delta_\ell + g_\eps(\ell)$}. Each of these processes, in virtue of Theorem~\ref{thm:Vervoort00Z+},  is recurrent. On the other hand, what is {somewhat} surprising is that  the process $\X$ with FTR $f(\ell, x) = \delta_\ell  f(0, x)$ is transient, in virtue of Theorem~\ref{thm:ACT18-main}, part 2).
\end{example}
\begin{remark} We emphasize the fact that outside the intervals $ \alpha \in (1/2, 1]$ and $\rho \in [0,\infty)$ the behaviour of the process is known from previous results (see Table I and Figure 1). Moreover, the proofs of Theorem~\ref{thm:ACT18-main} parts 3) and 4) cover more general cases, as stated in Propositions \ref{moreg1} and \ref{moreg2}.
{In principle parts 1) and 2) can also be adapted to more general reinforcements, e.g. with a slowly varying factor. To be more precise, our proofs rely on some integral estimations of series. In this context, reinforcements which are power functions are easy to deal with and give explicit estimates. On the other hand, the method itself covers more general cases. }
\end{remark}

\section{Proof of Theorem~\ref{thm:ACT18-main} }
Let $\tau := \inf\{ n>0 : X_n = 0 \}$ and 
\[ \displaystyle M_n = \sum_{x=0}^{X_{n \wedge \tau}-1} \dfrac{1}{w_n(x)} \quad \mbox{for $n \in \N$}, \mbox{ with $M_0= 0$}. \]
The process $\M = (M_n\colon   n \in \Z_+)$ is in general a non-negative supermartingale and will play a major role in our proofs. In fact, in virtue of our assumption that  ${\boldsymbol \delta}$ is DT, we have that $\M$ is  indeed  a martingale (see Lemma 3.0 in \cite{Davis90} for details).

{
We use the following 0-1 law (see Sellke \cite{Sellke94} or Takeshima \cite{Takeshima01} for a proof).}
{
\setcounter{theoAlph}{18}
\begin{theoAlph}[Sellke's 0-1 law] \label{Sellke} Consider the ERRW $\X$ on $\Z_+$. If $\Phi_x  = +\infty $  for all $x \in \mathbb{Z}_+$,  then $\X$ is  either  recurrent  a.s. or transient a.s..
\end{theoAlph}
}

\subsection{Proof of Theorem~\ref{thm:ACT18-main} part 1)} {Since Theorem \ref{thm:Vervoort00Z+} part i) covers the case  $\alpha=1$ (that is $\rho = 0$), hereafter we assume that $\alpha \in (1/2,1)$.}
For $x \in \N$, let
\[ N_x := \sum_{n=1}^{\tau-1} \1_{(X_n,X_{n+1})=(x,x-1)}. \]
Define the event  $E:=\big\{ \mbox{$\tau=+\infty$,\, $\displaystyle \lim_{n \to \infty} X_n=+\infty$} \big\}$, {which implies transience.} We reason by contradiction  and suppose that $\P(E)>0$. On $E$, we have  that for any $y \in \Z_+$, there exists a $n_y \in \N$ such that $w_{n_y}(x) = w_{\infty}(x)$ for all $x \le y$. This implies that  for all $n \ge n_y$, we have
$$M_n \ge \sum_{x= 0}^{y} \frac 1{w_{\infty}(x)}, \qquad \mbox{ on $E$.}$$
By taking limits, we have that on $E$ 
\begin{equation}
{\label{M} M_{\infty} \ge \sum_{x=0}^{\infty} \dfrac{1}{w_{\infty}(x)} 
= \sum_{x=0}^{\infty} \dfrac{1}{(N_x+1)^{\rho} (x+1)^{\alpha}}. 
}
\end{equation}
On the other hand, $\M$ is a non-negative martingale. Combining  Doob\rq{}s convergence theorem (see \cite{Williams91})  with  \eqref{M}, we have that  
\[ \sum_{x=0}^{\infty} \dfrac{(N_x+1)^{-\rho}}{(x+1)^{\alpha}} <+\infty, \qquad \mbox{on the event $E$.}\]
At the same time, as $\M$ is a non-negative martingale, we can apply 
 Austin's theorem (see \cite{Austin66}), {which says}
\[S^2(M):={\sum_{n=0}^{\infty}} (M_{n+1}-M_n)^2<+\infty, \qquad \qquad \mbox{a.s.}. \]
{
Since
\begin{align*}
M_{n+1}-M_n=\begin{cases}
\dfrac{1}{w_n(x)} &\mbox{if $(X_n,X_{n+1})=(x,x+1)$}, \\
-\dfrac{1}{w_n(x)} &\mbox{if $(X_n,X_{n+1})=(x+1,x)$} \\
\end{cases}
\end{align*}
for $n < \tau$ and $x \in \Z_+$, we have
\begin{align*}
 S^2(M) &= \sum_{x=0}^{\infty} \sum_{n=0}^{\infty} \dfrac{1}{w_n(x)^2} \1_{ \{X_n,\,X_{n+1}\} = \{x,\,x+1\} } = \sum_{x=0}^{\infty} \sum_{\ell=0}^{\phi_{\infty}(x)} \dfrac{1}{f(\ell,x)^2} \\
 &\geq \sum_{x=0}^{\infty} \sum_{k=0}^{N_x} \dfrac{1}{({\delta_{2k}} \cdot f(0,x))^2} \\
 &= \sum_{x=0}^{\infty} \sum_{k=0}^{N_x} \dfrac{1}{(k+1)^{2\rho}  (x+1)^{2\alpha}}, \qquad \mbox{a.s. on $E$},
\end{align*}
}
which implies
\begin{align*}
\sum_{x=0}^{\infty} \dfrac{(N_x+1)^{1-2\rho}}{(x+1)^{2\alpha}}<+\infty, \qquad {\mbox{a.s. on $E$}.}
\end{align*}

Let
\[ p={\dfrac{\alpha}{2\alpha-1}} \in (1,+\infty),\quad q={\dfrac{\alpha}{1-\alpha}} \in (1,+\infty) \]
so that
\[ \dfrac{1}{p}+\dfrac{1}{q}=1,\quad \dfrac{1}{p}+\dfrac{2}{q} = {\dfrac{1}{\alpha}}. \]
By H\"older's inequality,
\begin{align*}
\sum_{x=0}^{\infty} {\dfrac{(N_x+1)^{-\rho/p + (1-2\rho)/q}}{x+1}} \leq \left(\sum_{x=0}^{\infty} \dfrac{(N_x+1)^{-\rho} }{(x+1)^{\alpha}}\right)^{1/p} \left(\sum_{x=0}^{\infty} \dfrac{(N_x+1)^{1-2\rho}}{(x+1)^{2\alpha}}\right)^{1/q} <+\infty.
\end{align*}
On the other hand, if {$\rho \leq 1-\alpha$}, then we have $-\rho/p + (1-2\rho)/q \geq 0$ and
\[ 
\sum_{x=0}^{\infty} {\dfrac{(N_x+1)^{-\rho/p + (1-2\rho)/q}}{x+1} }
\geq \sum_{x=0}^{\infty} {\dfrac{1}{x+1}} = +\infty. \]
This gives  a contradiction, and proves that $\P(E) =0$. The result follows by Sellke\rq{}s 0-1 law (see {Theorem~\ref{Sellke}}). 

\subsection{Proof of Theorem~\ref{thm:ACT18-main} part 2)} {In this Section} we prove that $\X$ is transient, a.s., {under the assumptions of part 2)}. We reason by contradiction. Suppose that $\X$ is recurrent a.s. (see {Theorem~\ref{Sellke}}). We prove that $\M$  is bounded in $L^2$, which implies that $\M$ is uniformly integrable and $\P(M_{\infty} \neq 0)  >0$, which in turn implies transience.
In fact, for $\rho \in (0, 1/2)$, there exists a constant $c>0$ such that
\begin{equation}\label{mart1}
\E[M_n^2]  = \sum_{j=1}^n \E\left[(M_j - M_{j-1})^2\right] \le c \sum_{x=0}^\infty \frac{\E[(N_x+1)^{1 - 2\rho}]}{(x+1)^{2 \alpha}} \le  c \sum_{x=0}^\infty \frac{{(\E[N_x+1])^{1 - 2\rho}}}{(x+1)^{2 \alpha}}
\end{equation}
In the last step, we {used} Jensen\rq{}s inequality, as the map $y \mapsto y^{1- 2 \rho}$ is concave, for $\rho \in (0, 1/2)$.
Similarly, for $\rho =1/2$, we have
\begin{equation}\label{mart1.1}
 \E[M_n^2] \le \sum_{x=0}^\infty \frac{\ln\E[N_x+1]}{(x+1)^{2 \alpha}}.
 \end{equation}
In virtue of \eqref{mart1}, in order to prove Theorem~\ref{thm:ACT18-main} part 2), it is enough to prove that   there exists $\overline{C}$  such that 
\begin{equation}\label{mart2}
\E[N_x] \le \overline{C} x^{\frac{2}{1- \rho} }
\end{equation}
for all $x \in \N$.
In order to see why \eqref{mart2} is sufficient for our purposes, simply notice  that 
$  2\alpha  - 2(1- 2\rho)/(1-\rho) > 1$ {is equivalent to the condition of the Theorem.}
{The remaining part of this Section is devoted to prove \eqref{mart2}. In particular, Lemma~\ref{martle3} below is the key result for our goal.}
\begin{definition}\label{genpo} Fix $x \in \N$ and set $\gamma_x \Def (x+1)^\alpha/x^\alpha$. Consider a generalised P\'olya  urn, which initially contains one white and one black ball. The reinforcement scheme for white balls is $f^{\ssup w}(k) = k^\rho$, for $k \in \N$.  The reinforcement scheme for black balls is $f^{\ssup b}(k) = \gamma_x  k^\rho$, for $k \in \N$.  In other words, if the composition of the urn at stage $n$ is $z$ white balls and $y$ black balls, then the probability to pick a white ball at the next stage is $f^{\ssup w}(z)/(f^{\ssup w}(z)+ f^{\ssup b}(y))$. At each stage a ball is picked, and returned to the urn together  with another ball of the same colour.
Denote by $\P^x$ the measure describing this model, and {by $\E^x$ the expected value with respect to $\P^x$}. Denote by $(W_n, B_n)$, with $n \in \Z_+$, the composition of the urn by time $n$,  with $W_0 = B_0 = 1$.  
Denote by $\mathbf{Po}^{\ssup x}$ the sequence $(W_n, B_n\colon n\in \Z_+)$ under the measure $\P^x$.
\end{definition}
\begin{lemma}\label{martle3} {Assume that $\alpha \in (1/2,1]$ and $\rho \in (1-\alpha,1/2]$.} Let  
$H_n \Def \inf\{k \in \N \colon W_k =n\}$ for $n \in \N$. Define {$B^*_n \Def B_{H_n}.$} There exists a constant $C$ such that for any $x,n\in \N$, we have
\begin{equation}\label{mart3}
\E^{x}[B^*_n] \le \gamma_{x}^{\frac 1{1- \rho}} n   +  C n^{\frac{1+\rho}2} .
\end{equation}
\end{lemma}
\begin{proof} Consider Rubin\rq{}s embedding (see the Appendix of \cite{Davis90}), which is shortly described as follows. Let $(Y_i)_i$ and $(Z_i)_i$ be two independent sequences of independent exponentials with parameter one. Set, for each $n \in \N$,  
\begin{equation}\label{RubC}
 \widetilde{W}_n \Def \sum_{k=1}^n \frac{Y_k}{k^\rho}, \qquad \mbox{and} \qquad \widetilde{B}_n \Def  \frac 1{\gamma_x} \sum_{k=1}^n \frac{Z_k}{ k^\rho}.
\end{equation}
The variables $(\widetilde{W}_n, \widetilde{B}_n\colon n \in \N)$ can be used to generate a P\'olya urn process with the features of Definition~\ref{genpo}.  In this context,  $\{\widetilde{B}_s <  \widetilde{W}_n\}$ for  $s, n \in \N$ if and only if by the time the urn contains $n+1$ white balls, it  contains at least $s+1$ black ones. 
Let 
$$a_{n, x} \Def   {\left\lceil n \cdot \gamma_x^{\frac {1}{1 -\rho}}\right\rceil,} $$
{where   $\lceil x \rceil$ denote the smallest integer larger or equal to $x$.}
Fix a sequence  $\theta_n \in (0,1/2)${, which will be specified later}. For $s\ge a_{n, x}$ we have 
$$
\begin{aligned}
 \P^x(\widetilde{B}_{s} \le \widetilde{W}_n)&=  \P^x({\rm e}^{ \theta_n \widetilde{B}_{s}}  \le {\rm e}^{\theta_n \widetilde{W}_n})\\
&\le  \E^x\left[{\rm e}^{ \theta_n \widetilde{W}_n}\right]\E^x\left[{\rm e}^{ - \theta_n \widetilde{B}_{s}}\right] \qquad (\mbox{Markov ineq.})\\
&= \prod_{k=1}^n \frac 1{1 - \theta_n/k^{\rho}}\prod_{j=1}^{a_{n, x}} \frac 1{1+ \theta_n/(\gamma_x j^{\rho})}\prod_{t=a_{n, x}+1}^s \frac 1{1+ \theta_n/(\gamma_x t^{\rho})}.
\end{aligned}
$$
Call the product of the first two terms $I_{\theta_n, n,x}$ and the third term $II_{\theta_n, n,x, s}$. We have 
$$
\begin{aligned}
I_{\theta_n, n,x} &\le \exp\left\{ \theta_n \sum_{j=1}^n \frac{j^{- \rho}}{1 - \theta_n/j^{\rho}} -  \theta_n \sum_{j=1}^{a_{n,x}} \frac{\frac 1{\gamma_x j^{ \rho}}}{1 + \theta_n/(\gamma_x j^{\rho})}\right\} \\
&\le \exp\left\{ \frac{\theta_n}{1- \theta_n} \sum_{j=1}^n j^{- \rho} -  \frac{\theta_n}{1+\frac{\theta_n}{\gamma_x}} \frac 1{\gamma_x}\sum_{j=1}^{a_{n,x}} j^{- \rho}\right\}\\
&\le \exp\left\{ \frac{\theta_n}{(1- \rho)(1- \theta_n)} n^{1- \rho} -  \frac{\theta_n}{(1- \rho)(1+\frac{\theta_n}{\gamma_x})}  \frac 1{\gamma_x}(a_{n, x}^{1-\rho} -1)\right\}\\
&\le \exp\left\{ \frac{\theta_n}{(1- \rho)(1- \theta_n)} n^{1- \rho} -  \frac{\theta_n}{(1- \rho)(1+\frac{\theta_n}{\gamma_x})}   n^{1-\rho}+ \frac{\theta_n}{(1- \rho)(\gamma_x+\theta_n)}  \right\}\\
&\le \exp\left\{ \frac{{2n^{1-\rho} \theta_n^2}}{(1- \rho)(1- \theta_n)(1+\frac{\theta_n}{\gamma_x})}  + \frac{\theta_n}{(1- \rho)(1+\theta_n)}  \right\}.
 \end{aligned}
 $$
{The first inequality follows from an elementary bound
\[ 
 \exp\left(\dfrac{x}{1+x}\right) \leq 1+x \qquad \mbox{for $x>-1$.}
\]
The third inequality uses an integral comparison.
We can obtain the fourth and fifth inequalities by noting that $a_{n,x} \geq n \gamma_x^{\frac{1}{1-\rho}}$ and $\gamma_x \geq 1$, respectively.
} On the other hand,
 $$
 \begin{aligned}
 II_{\theta_n, n,x,s}&\le \exp\left\{- \frac{\theta_n}{{\gamma_x}+ \theta_n} \sum_{j = a_{n,x}+1}^{s} j^{-\rho}\right\}\\
 &\le \exp\left[- {2C_1}  \theta_n \Big\{s^{1 -\rho} -   (a_{n, x}+1)^{1-\rho}\Big\}\right].
  \end{aligned}
$$ 
By choosing $\theta_n = {(1/2)}n^{-(1-\rho)/2}$, {we can see that $I_{\theta_n, n,x}$ is bounded by a positive constant $C_2$. Thus,} we have
$$
\begin{aligned}
&\sum_{s = a_{n, x} +1}^\infty \P^x(\widetilde{B}_{s} \le \widetilde{W}_n) \\
&\leq C_2\exp\Big\{ C_1 n^{-\frac{1-\rho}2}(a_{n, x}+1)^{1-\rho}\Big\}\sum_{s = a_{n, x} +1}^{\infty}\exp\Big(- C_1 n^{-\frac{1-\rho}2} s^{1 -\rho} \Big) \\
&\leq C_2\exp\Big\{ C_1 n^{-\frac{1-\rho}2}(a_{n, x}+1)^{1-\rho}\Big\} \int_{a_{n, x}}^{\infty}\exp\Big(- C_1 n^{-\frac{1-\rho}2} s^{1 -\rho} \Big) \d s.
\end{aligned}
$$
{Suppose that $n$ is large enough to imply $C_1 n^{-\frac{1-\rho}{2}} a_{n, x}^{1 -\rho} \geq 1$.}
Letting $t= C_1 n^{-\frac{1-\rho}2} s^{1 -\rho} $ and noting that $\rho \in (0,1/2]$,
$$
\begin{aligned}
&\int_{a_{n, x}}^{\infty}\exp\Big(- C_1 n^{-\frac{1-\rho}2} s^{1 -\rho} \Big) \d s \\
&= (1-\rho)^{-1} \big( C_1 n^{-\frac{1-\rho}2} \big)^{-\frac{1}{1-\rho}} \int_{C_1 n^{-(1-\rho)/2} a_{n, x}^{1 -\rho}}^{\infty}t^{\frac{\rho}{1-\rho}}e^{-t} \d t \\
&\leq  C_3 n^{\frac{1}{2}} \Big(C_1 n^{-\frac{1-\rho}{2}} a_{n, x}^{1 -\rho}\Big)^{\frac{2\rho-1}{1-\rho}} \int_{C_1 n^{-(1-\rho)/2} a_{n, x}^{1 -\rho}}^{\infty}te^{-t} \d t \\
&= C_3 n^{\frac{1}{2}} \Big(C_1 n^{-\frac{1-\rho}{2}} a_{n, x}^{1 -\rho}\Big)^{\frac{2\rho-1}{1-\rho}}  \Big(C_1 n^{-\frac{1-\rho}{2}} a_{n, x}^{1 -\rho}+1\Big) \exp \Big(-C_1 n^{-\frac{1-\rho}{2}} a_{n, x}^{1 -\rho}\Big) \\
&\leq C_4 n^{\frac{1+\rho}{2}}\exp \Big(-C_1 n^{-\frac{1-\rho}{2}} a_{n, x}^{1 -\rho}\Big).
\end{aligned}
$$
Hence we have
$$
\begin{aligned}
&\sum_{s = a_{n, x} +1}^\infty \P^x(\widetilde{B}_{s} \le \widetilde{W}_n) \\
&\leq {C_2}C_4 n^{\frac{1+\rho}{2}} \exp \left[ C_1 n^{-\frac{1-\rho}{2}} \big\{(a_{n, x}+1)^{1 -\rho}- a_{n, x}^{1 -\rho}\big\}\right] \leq {C_5} n^{\frac{1+\rho}{2}}.
\end{aligned}
$$
In  the last inequality, we used the fact that for $m \in \N$,  we have
\[ (m+1)^{1-\rho}  - m^{1-\rho} \leq (1-\rho) m^{-\rho}. \]
The previous inequality can be proved via the mean value theorem applied to the function $h(t) = (m+t)^{1-\rho}$, defined for $t >0$.
Finally we have 
$$
\begin{aligned}
\E^x[B^*_{n}] &= \sum_{k = 1}^{a_{n,x}}  \P^x(\widetilde{B}_{k} \le \widetilde{W}_n) +  \sum_{s = a_{n, x} +1}^\infty \P^x(\widetilde{B}_{s} \le \widetilde{W}_n)
\le {\gamma_x^{\frac 1{1-\rho}} n + 1 + C_5 n^{\frac{1+\rho}{2}}}
\end{aligned}
$$
{for all large $n$. By choosing a large $C>0$, we obtain 
$$\E^x[B^*_{n}] \leq \gamma_x^{\frac 1{1-\rho}} n +  C n^{\frac{1+\rho}{2}} $$
for all $n$.}
\hfill
\end{proof}

We can use a collection of independent generalized P\'olya urns ($\mathbf{Po}^{\ssup x} \colon x \in {\N}$), where  $\mathbf{Po}^{\ssup x}$ has distribution $\P^x$, to generate a reinforced random walk $\X$. In this context,  the jumps from vertex $x$ are  modelled using the  urn $\mathbf{Po}^{\ssup x}$. Each time the process is at $x$, we pick a ball from the urn and observe its color. If it is black the walk moves to $x+1$, and moves to $x-1$ otherwise.
Recall that $N_x$ is the total number of jumps from $x$ to $x-1$ before time $\tau$.
As we assume that $\X$ is recurrent  a.s., the variable $N_x$ is $\sigma({\bf Po}^{\ssup k}\colon k \in \{1, 2, \ldots, x-1\})$-measurable. Therefore $N_x$ is independent of ${\bf Po}^{\ssup x}$. Using Lemma~\ref{martle3} with the urn  $\mathbf{Po}^{\ssup x}$, with $n = N_x$, we have
\begin{equation}\label{eq:condvers}
{\E}[N_{x+1}\;|\; N_{x}] \le \gamma_x^{\frac 1{1-\rho}} N_x + C N_x^{\frac{1+\rho}2}.
\end{equation}
As $\rho \in (0, 1/2]$, we have that $(1+\rho)/2<1$.
By taking the expected value of both sides in \eqref{eq:condvers} and using Jensen\rq{}s inequality, we have that 
\begin{equation}\label{eq:condvers1}
{\E}[N_{x+1}] \le \gamma_x^{\frac 1{1-\rho}} {\E} [N_x] + C {\E} [N_x]^{\frac{1+\rho}2}.
\end{equation}
\begin{proof}[Proof of Theorem~\ref{thm:ACT18-main} part 2)]
Consider a sequence $(a_x \colon x \in \N)$ satisfying
\begin{equation}\label{eq:recur}
a_{x+1} \le  \gamma_x^{\frac 1{1-\rho}} a_x + C a_x^{\frac{1+\rho}2}, \qquad \mbox{with $a_1<\infty$.}
\end{equation}
{Notice that the sequence $a_x := \E[N_x]$ satisfies \eqref{eq:recur}, and $a_1 =1$ is finite in virtue of the definition of $N_1$.  Hence \eqref{mart2} is  proved once 
we prove that}
\begin{equation} \label{eq:seqprop}
a_x \le \overline{C} x^{\frac{2}{1-\rho}},
\end{equation}
for some positive constant $\overline{C}$,  and all $x \in \N$. We prove this by induction.  Of course it is true for $x =1$, as we can choose simply $\overline{C}$ large enough. Suppose it is true for $x$.
Using \eqref{eq:recur}, we have that 
\begin{equation}\label{eq:recur1}
a_{x+1}\le \overline{C}\gamma_x^{\frac 1{1-\rho}} x^{\frac{2}{1-\rho}} +  C \overline{C}^{\frac{1+\rho}2}x^{\frac{1+\rho}{1-\rho}}.
\end{equation}
Hence
\begin{equation}\label{eq:recur2}
\frac{a_{x+1}}{\overline{C} (x+1)^{\frac{2}{1-\rho}}} \le \gamma_x^{\frac 1{1-\rho}} \left(\frac x{x+1}\right)^{\frac{2}{1-\rho}} + C \overline{C}^{ \frac{\rho-1}2} \left(\frac x{x+1}\right)^{\frac{1+\rho}{1-\rho}}\frac{1}{x+1}.
\end{equation}
Set $\widetilde{C} = C (\overline{C})^{ (\rho-1)/2}$. 
{Notice that as $\rho \in (0, 1/2]$, the larger $\overline{C}$ is, the smaller $\widetilde{C}$ becomes,} approaching zero in the limit. 
Using $\alpha, \rho \le 1$,
the right-hand side of \eqref{eq:recur2} can be bounded as follows:
\begin{equation}\label{eq:recur2.1}
\begin{aligned}
&\le \left(\frac{x+1}{x}\right)^{\frac{1}{1-\rho}} \left(\frac x{x+1}\right)^{\frac{2}{1-\rho}} + \widetilde{C}  \left(\frac x{x+1}\right)^{\frac{1+\rho}{1-\rho}} \frac{1}{x+1}\\
&=  \left(\frac{x}{x+1}\right)^{\frac{1}{1-\rho}}\left\{1  + \widetilde{C} \left(\frac{x}{x+1}\right)^{\frac{\rho}{1-\rho}} \frac1{x+1}\right\}\\
&\le \left(\frac{x}{x+1}\right)^{\frac{1}{1-\rho}}  \left(1 + \frac{\widetilde{C}}{x+1}\right).
\end{aligned}
\end{equation}
We can choose  $\widetilde{C}$ smaller than  1  (i.e. $\overline{C}$ large enough). Hence
$$
\begin{aligned}
\mbox{the right-hand side of \eqref{eq:recur2.1}} &\le  \left(\frac{x}{x+1}\right)^{\frac{1}{1-\rho}}  \left(1 + \frac{1}{x+1}\right) \\
&\le  \left(\frac{x+1}{x+2}\right)^{\frac{1}{1-\rho}}  \left(\frac{x+2}{x+1}\right) = \left(\frac{x+1}{x+2}\right)^{\frac{\rho}{1-\rho}}\le 1.
\end{aligned}
$$

\hfill
\end{proof}

\subsection{Proof of Theorem~\ref{thm:ACT18-main} part 3)}
We prove a more general result, and the proof is closely related to the one given by Davis  \cite{Davis89}.
\begin{proposition}\label{moreg1} Suppose that $\mathbf{X}$ has FTR and $\boldsymbol{\delta}$ is DT. {Assume that $F_0^{\ssup 2}<+\infty$ and $F_0^{\ssup 1}=+\infty$.} If $ \sum_{k=0}^{\infty} \delta_{2k}^{-2}<+\infty$ and $\sum_{k=0}^{\infty} \delta_{2k}^{-1}=+\infty$, then $\mathbf{X}$ is transient a.s..
\end{proposition}

\begin{proof} {For each $x \in \mathbb{Z}_+$,}  
\begin{align*}
Z_x&:=\sum_{n=0}^{\infty} (M_{n+1}-M_n)^2 \cdot \1_{\{X_n,X_{n+1}\} = \{ x,x+1\}} \le  \sum_{\ell=0}^{\infty} \dfrac{1}{f(\ell,x)^2} \\
&= \sum_{k=1}^{\infty} \left\{ \dfrac{1}{f(2(k-1),x)^2} +  \dfrac{1}{f(2k-1,x)^2} \right\} \\
&\le \dfrac{1}{f(0,x)^2} + 2\sum_{k=1}^{\infty} \dfrac{1}{\{\delta_{2k} \cdot f(0,x)\}^2}= \left( 1 +2\sum_{k=1}^{\infty} \dfrac{1}{(\delta_{2k})^2} \right) \cdot \dfrac{1}{f(0,x)^2}.
\end{align*}
Thus we have
\begin{align*}
S^2(M) = \sum_{x=0}^{\infty} Z_x &\le \left( 1 +2\sum_{k=1}^{\infty} \dfrac{1}{(\delta_{2k})^2} \right) \cdot \left( \sum_{x=0}^{\infty} \dfrac{1}{f(0,x)^2} \right)=:\lambda <+\infty
\end{align*}
with probability one. {By the orthogonality of martingale increments,} we have
\[ \E[(M_N-M_1)^2] = \sum_{n=1}^{N-1}\E\left[ (M_{n+1}-M_n)^2\right] \leq \lambda, \]
for any $N$, which shows that $\{ M_n \}$ is an $L^2$-bounded martingale. We have
\[  \E\left[\lim_{n \to \infty} M_n \right] = \lim_{n \to \infty} \E[M_n] = M_1 =\dfrac{1}{f(0,0)} > 0,  \]
which implies
\[ \P(\tau<+\infty) \leq \P\left(  \lim_{n \to \infty} M_n=0 \right)<1. \]
This together with Sellke's 0-1 law (see {Theorem~\ref{Sellke}}) shows that $\mathbf{X}$ is transient.~\hfill\end{proof}

\subsection{Proof of Theorem~\ref{thm:ACT18-main} part 4)}
We provide a proof for a  more general result, which includes initially transient cases.
\begin{proposition}\label{moreg2} Suppose that $\mathbf{X}$ has FTR, and there exists a constant $C \in (0, \infty)$ such that
\begin{equation}
\dfrac{f(0,x)}{f(0,x-1)} \leq C \quad \mbox{for all $x\in \N$.} 
\end{equation}
If $ \sum_{\ell=0}^{\infty}{\delta_{\ell}^{-1}}<+\infty$, then 
 $\mathbf{X}$  localizes on a single edge a.s..
\end{proposition}

\begin{proof} 
For each $x\in \N$, we define 
$$E_x:= \left\{ \sum_{n=0}^\infty \1_{(X_n, X_{n+1}) = (x, x+1)} = 0\right\},$$
that is the event that the process never jumps from $x$ to $x+1$. The $k$-th time the process visits vertex $x$, {the conditional probability} that it jumps to $x-1$ is 
$$  \dfrac{{\delta_{2k - 1}} \cdot f(0,x-1)}{{\delta_{2k - 1}} \cdot f(0,x-1)+f(0,x)}.$$
Then, we have 
\begin{align*}
\P(E_x) 
&{\ge}\prod_{k=1}^{\infty} \dfrac{{\delta_{2k - 1}} \cdot f(0,x-1)}{{\delta_{2k - 1}} \cdot f(0,x-1)+f(0,x)} \\
&= {\prod_{k=1}^{\infty}} \left(1+ \dfrac{f(0,x)}{{\delta_{2k - 1}} \cdot f(0,x-1)} \right)^{-1} \\
&\geq \exp\left( -\dfrac{f(0,x)}{f(0,x-1)} {\sum_{k=1}^{\infty} \dfrac{1}{\delta_{2k-1}}} \right) 
\geq \exp\left( -C {\sum_{k=1}^{\infty} \dfrac{1}{\delta_{2k-1}}} \right)>0.
\end{align*}
This shows that $\sum_{x} \P(E_x) =+\infty$. The second Borel-Cantelli lemma implies that $\P(\mbox{$E_x$ occurs for infinitely many $x$'s})=1$ and {$\P(\mbox{$\mathbf{X}$ is of finite range})=1$.} In fact we have $\P(\mbox{$\mathbf{X}$ is localized to a single edge})=1$ by an application of Rubin's theorem (see Corollary 3.6 in \cite{Takeshima00}). \hfill 
\end{proof}

\section{Appendix}\label{appendix}
\begin{proof}[{\bf Proof of Theorem \ref{thm:Vervoort00Z+}}] 
Let $E:=\left\{ \mbox{$\tau=+\infty$,\, $\displaystyle \lim_{n \to \infty} X_n=+\infty$} \right\}$. {By Theorem \ref{Sellke}, $\P(E)=0$ implies that $\mathbf{X}$ is recurrent a.s..} 
Recall the definition of $\phi$ from \eqref{def:phi}. On the event $E$, 
\[ {M_{\infty} \geq \sum_{x=0}^{\infty} \dfrac{1}{w_{\infty}(x)} = \sum_{x=0}^{\infty} \dfrac{1}{\delta_{\phi_{\infty}(x)} \cdot f(0,x)} .} \]
Suppose that $\delta_k \leq K$ for all $k \in \Z_+$. Then we have
\[ \sum_{x=0}^{\infty} \dfrac{1}{\delta_{\phi_{\infty}(x)} \cdot f(0,x)} \geq \dfrac{1}{K}\sum_{x=0}^{\infty} \dfrac{1}{f(0,x)}=+\infty. \]
{By Doob\rq{}s convergence theorem,} $\P(E)$ cannot be positive.

Next we assume ii). Define 
\[ \Theta_n := M_n + \sum_{m=1}^{n} \left\{ \dfrac{1}{w_m (X_m )} - \dfrac{1}{w_{m+1} (X_m )} \right\} \cdot \1_{X_{m \wedge \tau} < X_{(m+1) \wedge \tau}}. \]
The process $(\Theta_n\colon n \in \Z_+)$  is a nonnegative martingale. (see Lemma 3.0 in \cite{Davis90} for details). 
We rewrite
\[ \Theta_n = \sum_{x=0}^{X_{n \wedge \tau}-1}\dfrac{1}{f(0,x)} \sum_{\ell=0}^{\phi_n(x)-1} \dfrac{(-1)^{\ell}}{\delta_{\ell}}. \]
Fix $x \in \mathbb{Z}_+$ and suppose that $\phi_n(x)=2k+1$ for some $k \in \mathbb{Z}_+$, we have
\begin{align*}
\sum_{\ell=0}^{\phi_n(x)-1} \dfrac{(-1)^{\ell}}{\delta_{\ell}} &= \sum_{m=1}^{k-1} \left(\dfrac{1}{\delta_{2m}} - \dfrac{1}{\delta_{2m+1}}\right) + \dfrac{1}{\delta_{2k}}.
\end{align*}
Assume that $\delta_{2k_0} < \delta_{2k_0+1}$, and let $C := \dfrac{1}{\delta_{2k_0}} - \dfrac{1}{\delta_{2k_0+1}}>0$. 
Then we have
\begin{align*}
\sum_{m=1}^{k-1} \left(\dfrac{1}{\delta_{2m}} - \dfrac{1}{\delta_{2m+1}}\right) + \dfrac{1}{\delta_{2k}} \geq
\begin{cases}
\dfrac{1}{\delta_{2k_0}} - \dfrac{1}{\delta_{2k_0+1}} = C &\mbox{if $k_0<k$}, \vspace{2mm}\\
\dfrac{1}{\delta_{2k}} \geq \dfrac{1}{\delta_{2k_0}} \geq 
C &\mbox{if $k \leq k_0$},
\end{cases}
\end{align*}
and
\[ \Theta_n \geq C \sum_{x=0}^{X_{n \wedge \tau}-1}\dfrac{1}{f(0,x)}. \]
This shows that $\P(E)$ cannot be positive.
\hfill
\end{proof}

\begin{ack}
A.C. is grateful to Yokohama National University for its hospitality, and he was supported by ARC grant  DP180100613 and Australian Research Council Centre of Excellence for Mathematical and Statistical Frontiers (ACEMS) CE140100049.
M.T. is partially supported by JSPS Grant-in-Aid for Young Scientists (B) No. 16K21039.
The authors thank Ben Amiet for reading the manuscript and helping to produce Figure 1. {They also thank two anonymous referees for detailed comments. Finally they thank Amanoya for offering a very nice environment, where  part of this research was carried.}
\end{ack}

\end{document}